\documentclass[11pt]{amsart}
\usepackage[margin=1.35in]{geometry}
\usepackage{amscd,amsmath,amsxtra,amsthm,amssymb,stmaryrd,xr,mathrsfs,mathtools,enumerate,commath, comment, enumitem}
\usepackage{stmaryrd}
\usepackage{xcolor}
\usepackage{commath}
\usepackage{comment}
\usepackage{tikz-cd}
\usepackage{longtable} 
\usepackage{pdflscape} 
\usepackage{booktabs}
\usepackage{hyperref}
\definecolor{vegasgold}{rgb}{0.77, 0.7, 0.35}
\definecolor{darkgoldenrod}{rgb}{0.72, 0.53, 0.04}
\definecolor{gold(metallic)}{rgb}{0.83, 0.69, 0.22}
\hypersetup{
 colorlinks=true,
 linkcolor=darkgoldenrod,
 filecolor=brown,      
 urlcolor=gold(metallic),
 citecolor=darkgoldenrod,
 }
\newtheorem{lthm}{Theorem}

\usepackage[all,cmtip]{xy}
\usepackage[T1]{fontenc}

\DeclareFontFamily{U}{wncy}{}
\DeclareFontShape{U}{wncy}{m}{n}{<->wncyr10}{}
\DeclareSymbolFont{mcy}{U}{wncy}{m}{n}
\DeclareMathSymbol{\Sh}{\mathord}{mcy}{"58}
\usepackage[T2A,T1]{fontenc}
\usepackage[OT2,T1]{fontenc}

\newtheorem{theorem}{Theorem}[section]
\newtheorem{lemma}[theorem]{Lemma}

\newtheorem*{theorem*}{Theorem}
\newtheorem*{ass*}{Assumption}
\newtheorem{definition}[theorem]{Definition}
\newtheorem{corollary}[theorem]{Corollary}

\newtheorem{proposition}[theorem]{Proposition}

\newcommand{\cF}{\mathcal{F}}

\newcommand{\fl}{\mathfrak{l}}

\newcommand{\cU}{\mathcal{U}}

\newcommand{\cH}{\mathcal{H}}

\makeatletter
\DeclareFontFamily{OMX}{MnSymbolE}{}
\DeclareSymbolFont{MnLargeSymbols}{OMX}{MnSymbolE}{m}{n}
\SetSymbolFont{MnLargeSymbols}{bold}{OMX}{MnSymbolE}{b}{n}
\DeclareFontShape{OMX}{MnSymbolE}{m}{n}{
    <-6>  MnSymbolE5
   <6-7>  MnSymbolE6
   <7-8>  MnSymbolE7
   <8-9>  MnSymbolE8
   <9-10> MnSymbolE9
  <10-12> MnSymbolE10
  <12->   MnSymbolE12
}{}
\DeclareFontShape{OMX}{MnSymbolE}{b}{n}{
    <-6>  MnSymbolE-Bold5
   <6-7>  MnSymbolE-Bold6
   <7-8>  MnSymbolE-Bold7
   <8-9>  MnSymbolE-Bold8
   <9-10> MnSymbolE-Bold9
  <10-12> MnSymbolE-Bold10
  <12->   MnSymbolE-Bold12
}{}
\let\llangle\@undefined
\let\rrangle\@undefined
\DeclareMathDelimiter{\llangle}{\mathopen}%
                     {MnLargeSymbols}{'164}{MnLargeSymbols}{'164}
\DeclareMathDelimiter{\rrangle}{\mathclose}%
                     {MnLargeSymbols}{'171}{MnLargeSymbols}{'171}
\makeatother
\newcommand{\Z}{\mathbb{Z}}

\newcommand{\Q}{\mathbb{Q}}
\newcommand{\F}{\mathbb{F}}

\newcommand{\cC}{\mathcal{C}}
\newcommand{\cO}{\mathcal{O}}

\newcommand{\op}[1]{\operatorname{#1}}

\numberwithin{equation}{section}

\begin{document}

\title[Galois surjectivity for most Drinfeld modules]{Galois representations are surjective for almost all Drinfeld modules}

\author[A.~Ray]{Anwesh Ray}
\address[Ray]{Chennai Mathematical Institute, H1, SIPCOT IT Park, Kelambakkam, Siruseri, Tamil Nadu 603103, India}
\email{anwesh@cmi.ac.in}


\maketitle

\begin{abstract}
This article advances the results of Duke on the average surjectivity of Galois representations for elliptic curves to the context of Drinfeld modules over function fields. Let $F$ be the rational function field over a finite field. We establish that for Drinfeld modules of rank $r \geq 2$, the $T$-adic Galois representation:
$\hat{\rho}_{\phi, T}: \op{Gal}(F^{\op{sep}}/F) \rightarrow \op{GL}_r(\F_q\llbracket T\rrbracket)$ is surjective for a density $1$ set of such modules. The proof utilizes Hilbert irreducibility (over function fields), Drinfeld's uniformization theory and sieve methods.
\end{abstract}

\section{Introduction}
\subsection{Background and motivation} Elliptic curves naturally give rise to Galois representations, and their arithmetic is encoded in the Galois theoretic properties of these representations. Given an elliptic curve $E_{/\Q}$ and a natural number $n$, let $E[n]$ denote the $n$-torsion subgroup of $E(\bar{\Q})$, equipped with the natural action of the absolute Galois group $\op{G}_{\Q}:=\op{Gal}(\bar{\Q}/\Q)$. This action is encoded by the representation 
\[\rho_{E, n}: \op{G}_{\Q}\rightarrow \op{GL}_2(\Z/n\Z).\] There has been much interest in the study of the images of such representations, and the analysis and classification of these images is possible via the study of modular curves and their geometric properties, cf. for instance \cite{Mazurmodcurves,Sutherland, zywina}. Let $\hat{\rho}:\op{G}_{\Q}\rightarrow \op{GL}_2(\widehat{\Z})$ be the adelic Galois representation on the Tate module $\mathbb{T}(E):=\varprojlim_n E[n]$. Serre's open image theorem states that the image $\hat{\rho}$ has finite index in $\op{GL}_2(\widehat{\Z})$. A prime $\ell$ is \emph{exceptional} if $\rho_{E, \ell}$ is not surjective. Duke \cite{duke1997elliptic} showed that that most elliptic curves (ordered according to height) have no exceptional primes. Jones \cite{Jones} obtained certain refinements of this result, proving that the index $[\op{G}_{\Q}:\hat{\rho}(\op{G}_{\Q})]=2$ for most elliptic curves $E_{/\Q}$. 
\subsection{Main result}
\par Let $p$ be a prime number, $q$ a power of $p$ and $\F_q$ the finite field with $q$ elements. Set $A:=\F_q[T]$ and $F$ the rational function field $\F_q(T)$. A Drinfeld module over $A$ is an additive $A$-module structure on the additive group scheme over $F$. These objects were introduced by Drinfeld and give natural analogues of elliptic curves over function fields. Much like elliptic curves, they give rise to compatible families of Galois representations in characteristic $p$, and they occupy a central role in the Langlands program in positive characteristic. In this context, little work has been done with regard to the classification of Galois images. Pink and R\"utsche \cite{pinkrutsche} proved an analogue of Serre's open image theorem. Zywina \cite{zywina2011drinfeld} exhibited an explicit Drinfeld module of rank $2$ for which the adelic Galois representation is surjective. The goal of this article to prove that the $T$-adic Galois representations associated to Drinfeld modules are surjective on average.
\begin{lthm}[Theorem \ref{main thm end}]\label{main thm}
    Let $r\geq 2$ be an integer. Then the $T$-adic Galois representation 
    \[\hat{\rho}_{\phi, T}:\op{Gal}(F^{\op{sep}}/F)\rightarrow \op{GL}_r(\F_q\llbracket T\rrbracket)\]
    for a density $1$ set of Drinfeld modules $\phi$ over $A$ of rank $r\geq 2$. 
\end{lthm}
This notion of density is made precise in section \ref{s 4}. It was shown by the author in \cite{raytadic} that the set of Drinfeld modules of rank $2$ for which \[\hat{\rho}_{\phi, T}:\op{Gal}(F^{\op{sep}}/F)\rightarrow \op{GL}_2(\F_q\llbracket T\rrbracket)\]
has positive density. This result is an improvement since not only does it apply for all $r\geq 2$, but moreover the result is satisfied for almost all Drinfeld modules.

\subsection{Methodology}
\par The results in this article are proven via a synthesis of techniques from Galois theory and arithmetic statistics. A criterion of Pink of R\"utsche (cf. Proposition \ref{PR prop}) gives conditions for which $\hat{\rho}_{\phi, T}$ is surjective. In fact, it suffices to study the mod-$T^2$ reduction of $\hat{\rho}_{\phi, T}$. An analogue of the Hilbert irreducibility theorem over the rational function field \cite[Corollary 3.5]{HITFF} shows that the mod-$T$ representation is surjective for most Drinfeld modules of rank $r$ (see Proposition \ref{cc prime =1}). In order to study the mod-$T^2$ image, I introduce further local conditions at all primes $\mathfrak{l}$ of $F$ such that $\mathfrak{l}\neq \{(T), \infty\}$. These conditions are introduced via an in depth analysis of Drinfeld--Tate uniformizations. In section \ref{s 6}, standard sieve theoretic arguments are then used to prove Theorem \ref{main thm}. Here is a more concise single-paragraph version:

\subsection{Outlook} The study of Galois representations attached to Drinfeld modules suggests a number of distribution problems that merit further investigation. One natural step is to investigate the full adelic Galois image of Drinfeld modules of specific ranks. However, achieving such results remains challenging. Even if the residual Galois representation is surjective modulo two distinct primes $\mathfrak{p}$ and $\mathfrak{q}$, the image modulo $\mathfrak{p}\mathfrak{q}$ may still be strictly smaller than $\mathrm{GL}_r(A/\mathfrak{p}\mathfrak{q})$ due to arithmetic entanglements between the corresponding torsion fields. Put differently, the division fields attached to different primes need not be linearly disjoint, and controlling such intersections requires a much finer analysis than what is currently available. The argument of Jones proving that $100\%$ of elliptic curves over $\mathbb{Q}$ have surjective adelic Galois image cannot be adapted to the Drinfeld setting, since it relies on analytic estimates with no known analogue over global function fields. Following Duke, the method uses a large sieve to control exceptional primes, and its key input is an estimate for weighted sums of Hurwitz class numbers in arithmetic progressions, obtained via the Fourier expansion of half-integral weight modular forms together with the Ramanujan bound for their coefficients \cite[Lemma~3]{duke1997elliptic}. For Drinfeld modules, these methods do not carry over.
\par A related problem that has been of interest is that of constructing examples of Drinfeld modules of general rank $r$ with maximal Galois image. This question is still unresolved. Another potential direction of inquiry is to characterize the mod-$T$ (and $T$-adic) Galois images of Drinfeld modules over rational function fields.

\subsection*{Data availability} No data was analyzed in proving the results in the article.
\subsection*{Conflict of Interest} There is no conflict of interest that the author wishes to report.

\subsection*{Acknowledgment
}We would like to thank the referee for the excellent report.

\section{Basic notions}
\subsection{Notation}
Throughout, we let $p$ be an odd prime number and set $ q = p^n $. We denote the finite field with $ q $ elements by $\mathbb{F}_q$, and assume that $ q \geq 5 $. Let $ A = \mathbb{F}_q[T] $ be the polynomial ring in one variable over $\mathbb{F}_q$, and $ F = \mathbb{F}_q(T) $ its fraction field. A prime of $F$ is the maximal ideal in a discrete valuation ring $R\subset F$, whose fraction field is $F$. The primes can thus also be identified with the set of isomorphism classes discrete valuations of $F$. The set of primes is denoted by $\Omega_F$, let $\Omega_A\subset \Omega_F$ be the primes that arise from the non-zero prime ideals of $A$. The only other prime is $\infty$, corresponding to the valuation defined by setting $v_\infty(T)=-1$. When we write $\mathfrak{l}\in \Omega_A$, we think of $\mathfrak{l}$ as a non-zero prime ideal in $A$. On the other hand, when writing $v_{\mathfrak{l}}$, we refer to the associated valuation. Note that $\mathfrak{l}$ is principal, so is generated by a monic irreducible polynomial $a_{\mathfrak{l}}$. The valuation $v_{\mathfrak{l}}$ will be normalized by setting $v_{\mathfrak{l}}(a_{\mathfrak{l}}):=1$. Given $\mathfrak{l}\in \Omega_A$, let $A_\mathfrak{l}$ be the completion of $A$ at $\mathfrak{l}$ and set $F_\mathfrak{l}$ to be the fraction field of $A_\mathfrak{l}$. Take $\mathfrak{m}_\mathfrak{l}$ to be the maximal ideal of $A_\mathfrak{l}$ consisting of all elements with positive valuation. Set $k_\mathfrak{l}$ to denote the residue field $A_\mathfrak{l}/\mathfrak{m}_\mathfrak{l}$. For a non-zero ideal $\mathfrak{a}$ in $ A $ generated by a polynomial $ a $, the degree of $\mathfrak{a}$ is defined as $\deg \mathfrak{a} = \deg_T(a) $. One notes that $\deg \mathfrak{a}=\dim_{\F_q} (A/\mathfrak{a})$. 

 \subsection{Drinfeld Modules and $\F_q$-linear polynomials}
\par Drinfeld modules extend the classical theory of elliptic curves to function fields. In this subsection, we introduce these objects and the basic properties of their associated Galois representations. For a more detailed exposition, please refer to \cite{papibook}.

\begin{definition}
    An $ A $-field is a field $ K $ with an $\mathbb{F}_q$-algebra homomorphism $\gamma: A \rightarrow K$. The $ A $-characteristic of $ K $ is defined by:
\[
\op{char}_A(K) = 
\begin{cases}
0 & \text{if } \gamma \text{ is injective}, \\
\ker \gamma & \text{otherwise}.
\end{cases}
\]
If $\gamma$ is injective, $ K $ is said to be of generic characteristic.
\end{definition}
We set $ K^{\op{sep}} $ to denote a separable closure of $ K $, and $\op{G}_K = \op{Gal}(K^{\op{sep}}/K) $ the absolute Galois group of $ K $. 
\par A Drinfeld module is a faithful $A$-module structure on the additive group scheme. In order to make this notion precise, we introduce the $\F_q$-algebra of non-commutative twisted polynomials over $K$. Given an $\mathbb{F}_q$-algebra $ K $, the ring of twisted polynomials $ K\{\tau\} $ is defined by the rules:
\begin{itemize}
    \item Addition is done term-wise.
    \item Multiplication is defined according to the rule $(a\tau^i)(b\tau^j) = ab^{q^i} \tau^{i+j}$.
\end{itemize}
For a twisted polynomial $ f(\tau) = \sum_{i=0}^d a_i \tau^i $, define the action on $ x $ by:
\[
f(x) = \sum_{i=0}^d a_i x^{q^i}.
\]
One thus gets a polynomial that is $\F_q$-linear, i.e., 
\begin{itemize}
    \item $f(x+y)=f(x)+f(y)$, 
    \item $f(c x)=c f(x)$ for $c\in \F_q$. 
\end{itemize}
One defines the \emph{height} $\op{ht}_\tau(f) $ and \emph{degree} $\op{deg}_\tau(f) $ of $ f(\tau) $. Writing \[f(\tau) = a_h \tau^h + a_{h+1} \tau^{h+1} + \dots + a_d \tau^d,\] where $ a_h, a_d \neq 0$, set
\[
\op{ht}_\tau(f) := h \quad \text{and} \quad 
\op{deg}_\tau f(\tau):=d.\]
Note that the degrees of $f(\tau)$ and $f(x)$ are related by $\op{deg}_x f(x) = q^{\op{deg}_\tau(f)}$. Consider the derivative map $\partial: K\{\tau\}\rightarrow K$ defined by 
\[\partial\left(\sum_{i=0} a_i \tau^i\right):=a_0,\] and note that $\partial(f)=\frac{df}{dx}$. A polynomial $ f(x) $ is separable if its derivative in $ x $ does not vanish. For twisted polynomials, this corresponds to the property that $\op{ht}_\tau(f) = 0$.

\begin{definition}
    A Drinfeld module of rank $ r \geq 1 $ over an $ A $-field $ K $ is an $\mathbb{F}_q$-algebra homomorphism $\phi: A \rightarrow K\{\tau\}$ such that:
    \begin{itemize}
        \item $\partial(\phi_a) = \gamma(a)$ for all $a \in A$,
        \item $\op{deg}_\tau(\phi_a) = r \cdot \op{deg}_T(a)$.
    \end{itemize}
Here, $\phi_T = T + g_1 \tau + g_2 \tau^2 + \dots + g_r \tau^r$ with $ g_r \neq 0 $.
\end{definition}
Given a tuple $\vec{g}=(g_1, \dots, g_r)\in K^r$ such that $g_r\neq 0$, one obtains a Drinfeld module $\phi^{\vec{g}}$ of rank $r$ defined by setting
\[\phi_T^{\vec{g}}=T + g_1 \tau + g_2 \tau^2 + \dots + g_r \tau^r.\]
Given a Drinfeld module $\phi$, there is a $\phi$-twisted $A$-module structure on $K$. For $\alpha\in K$ and $b\in A$, set $b\cdot \alpha:=\phi_b(\alpha)$. We denote the associated $A$-module structure by $^\phi K$. In fact, given any extension $L/K$ contained in $K^{\op{sep}}$, the same procedure gives a twisted $A$-module structure on $\phi$. 
\begin{definition}
    Let $\phi, \psi: A\rightarrow K\{\tau\}$ be a pair of Drinfeld modules over an $A$-field $K$. A morphism $u: \phi\rightarrow \psi$ is by definition a twisted polynomial $u\in K\{\tau\}$ such that $u \phi_a=\psi_a u$ for all $a\in A$. An isogeny is a non-zero morphism. 
\end{definition}
If there exists an isogeny $u: \phi\rightarrow \psi$, then $\phi$ and $\psi$ have the same rank (cf. \cite[Proposition 3.3.4]{papibook}). We set $\op{Hom}_K(\phi, \psi)$ to be the group of morphisms $u: \phi\rightarrow \psi$. When $\phi=\psi$, set $\op{End}_K(\phi):=\op{Hom}_K(\phi, \phi)$. Observe that $\phi_a\in \op{End}_K(\phi)$ for all $a\in A$. This makes $\op{End}_K(\phi)$ into an $A$-algebra, and we set $\op{End}_K^0(\phi):=F\otimes_A \op{End}_K(\phi)$. 
\par Suppose that $\op{char}_A(K)=0$, then $u\mapsto \partial(u)$ is an injective homomorphism 
\[\partial: \op{End}_K(\phi)\rightarrow K.\] In particular, this implies that $\op{End}_K(\phi)$ is a commutative ring and $\op{End}_K^0(\phi)$ is a field extension of $F$. Given Drinfeld modules $\phi$ and $\psi$ of rank $r$, the module of morphisms $\op{Hom}_K(\phi, \psi)$ is a free $A$-module of rank $\leq r^2$ (cf. \cite{Drinfeldoriginal} or \cite[Theorem 3.4.1]{papibook}).
\subsection{Galois representations}

Much like elliptic curves, Drinfeld modules naturally induce Galois representations. Consider a Drinfeld module $\phi$ of rank $ r \geq 1$ over the global function field $ F $. For a non-zero polynomial $ a \in A $, the roots of $\phi_a(x)$ form the set $\phi[a] \subset F^{\op{sep}}$. If $ b \in A $, the relation $\phi_a(\phi_b(x)) = \phi_{ab}(x)$. It is easy to check that $\phi[a]$ is an $A$-submodule of $^\phi F^{\op{sep}}$ equipped with a natural action of $\op{G}_F=\op{Gal}(F^{\op{sep}}/F)$ by $A$-linear automorphisms. For a non-zero ideal $\mathfrak{a}$ in $ A $, set $\phi[\mathfrak{a}] = \phi[a]$ where $ a $ is the monic generator of $\mathfrak{a}$. Then, $\phi[\mathfrak{a}] \cong (A/\mathfrak{a})^r$. The associated Galois representation is as follows
\[
\rho_{\phi, \mathfrak{a}}: \op{G}_F \rightarrow \op{Aut}_A(\phi[\mathfrak{a}]) \cong \op{GL}_r(A/\mathfrak{a}).
\]

The $\fl$-adic Tate module $ \mathbb{T}_\fl(\phi) $ is defined as the inverse limit:
\[
\mathbb{T}_\fl(\phi) = \varprojlim_n \phi[\fl^n].
\]
By choosing an $ A_\fl $-basis of $ \mathbb{T}_\fl(\phi) $, we have the associated $A_\fl$-adic Galois representation:
\[
\hat{\rho}_{\phi, \fl}: \op{G}_F \rightarrow \op{GL}_r(A_\fl),
\]
which can be identified with the inverse limit $\varprojlim_{n} \rho_{\phi, \fl^n}$.
\par Let $F(\phi[\mathfrak{a}])$ denote the field extension of $F$ \emph{cut out} by $\phi[\mathfrak{a}]$. More precisely, it is the fixed field of the kernel of $\rho_{\phi, \mathfrak{a}}$. The homomorphism $\rho_{\phi, \mathfrak{a}}$ factors through the quotient 
\[\op{Gal}(F(\phi[\mathfrak{a}])/F)=\op{G}_F/\op{ker}\left(\rho_{\phi, \mathfrak{a}}\right),\] and induces an injection 
\[\op{Gal}(F(\phi[\mathfrak{a}])/F)\hookrightarrow \op{GL}_r(A/\mathfrak{a}).\] Thus the Galois group $\op{Gal}(F(\phi[\mathfrak{a}])/F)$ can be identified with the image of $\rho_{\phi, \mathfrak{a}}$.

\par Given a Drinfeld module $\phi$ over $F$ of rank $r$ defined by \[\phi_T= T+g_1 \tau+\dots+ g_r \tau^r,\]
let $\psi$ be the Drinfeld module of rank $1$ defined by 
\[\psi_T=T+(-1)^{r-1}g_r \tau.\]
Let $\mathfrak{a}$ be a non-zero ideal in $A$. The \emph{Weil pairing} a $A$-multilinear, surjective, non-degenerate and alternating pairing 
\[W_{\phi, \mathfrak{a}}: \prod_{i=1}^r \phi[\mathfrak{a}]\rightarrow \psi[\mathfrak{a}]. \]
As a consequence, we have that $\op{det}\rho_{\phi, \mathfrak{a}}=\rho_{\psi, \mathfrak{a}}$, see \cite[Theorem 3.7.1]{papibook} for further details. Let $\phi$ be a Drinfeld module over $F$ and $\mathfrak{l} \in \Omega_A$. We denote by $\phi_{\mathfrak{l}}$ the localized Drinfeld module over $F_{\mathfrak{l}}$.
\begin{definition}
    We say that $\phi$ has \emph{stable reduction} at $\mathfrak{l}$ if there exists a Drinfeld module $\psi$ over $F_\mathfrak{l}$ that is isomorphic to $\phi_\mathfrak{l}$ with coefficients in $A_\mathfrak{l}$, such that the reduction
\[
\bar{\psi}: A \rightarrow k_\mathfrak{l}\{\tau\}
\]
is a Drinfeld module. The rank of $\bar{\psi}$ is called the \emph{reduction rank} of $\phi$ at $\mathfrak{l}$. If the reduction rank is $r$, then $\phi$ is said to have \emph{good reduction} at $\mathfrak{l}$.
\end{definition}
Now let $\fl, \mathfrak{p}\in \Omega_A$ be distinct primes and consider the associated Galois representation $\rho=\hat{\rho}_{\phi, \mathfrak{p}}: \op{G}_F\rightarrow \op{GL}_2(A_{ \mathfrak{p}})$ on the $\mathfrak{p}$-adic Tate-module $\mathbb{T}_{\mathfrak{p}}(\phi)$. Let $\op{I}_{\mathfrak{l}}$ be the inertia subgroup of $\op{G}_{F_{\mathfrak{l}}}$. Set $\rho_{|\mathfrak{l}}$ to denote the restriction of $\rho$ to $\op{G}_{F_{\mathfrak{l}}}$. Suppose that $\mathfrak{l}\neq \mathfrak{p}$ is a prime at which $\phi$ has good reduction. Then, $\rho$ is unramified at $\mathfrak{l}$, i.e., $\op{I}_{\mathfrak{l}}$ lies in the kernel of $\rho_{|\mathfrak{l}}$ (cf. \cite[section 6.1]{papibook}). 
\section{The Drinfeld--Tate uniformization}
\par In this section, we recall the Drinfeld--Tate uniformization. For a more detailed exposition, we refer to \cite[sections 5.1, 6.2 and 6.3]{papibook}. This will be used in this article to study the local structure of Galois representations at primes of stable reduction. Let $\mathfrak{l}\in \Omega_A$, set $K:=F_{\fl}$ and $\cO:=A_{\fl}$. Let $\mathbb{C}_K$ be the completion of the separable closure of $K$. Let $\mathbb{C}_K\llangle x \rrangle$ be the set of power series of the form 
\[f(x)=\sum_{n\geq 0} a_n x^{q^n}\]
with coefficients $a_n\in \mathbb{C}_K$. These power series are $\F_q$-linear, i.e., 
\[f(\alpha x+\beta y)=\alpha f(x)+\beta f(y),\] where $\alpha, \beta\in \F_q$ and indeterminates $x$ and $y$. On the other hand, let $\mathbb{C}_K\{\!\{ \tau\}\!\}$ be the ring of twisted power series $\mathbb{C}_K\{\!\{ \tau\}\!\}$ consisting of
\[f=\sum_{n=0}^\infty a_n \tau^n,\] where $a_n\in \mathbb{C}_K$. Addition is defined termise, and multiplication as follows
\[\left(\sum_{n=0}^\infty a_n \tau^n\right)\left(\sum_{n=0}^\infty b_n \tau^n\right)=\sum_{n=0}^\infty \left(\sum_{i=0}^n a_i b_{n-i}^{q^i}\right)\tau^n.\]
There is a natural isomorphism 
\[\mathbb{C}_K\{\!\{ \tau\}\!\}\xrightarrow{\sim} \mathbb{C}_K\llangle x \rrangle\] defined by \[f\mapsto f(x):=\sum_{n=0}^\infty a_n x^{q^n}.\]
Given a discrete $\F_q$-vector subspace of $\mathbb{C}_K$, let 
\begin{equation}\label{e Lambda definition} e_{\Lambda}(x):=x\prod_{\lambda\in \Lambda\backslash\{0\}} \left(1-\frac{x}{\lambda}\right).\end{equation} This is an entire function given by an $\F_q$-linear power series in $\mathbb{C}_K\llangle x \rrangle$ called the \emph{Carlitz-Drinfeld exponential} of $\Lambda$.
\par Let $\varphi: A\rightarrow K\{\tau\}$ be a Drinfeld module with good reduction. Assume without loss of generality that $\varphi$ takes values in $\cO$; we indicate this by simply writing $\varphi: A\rightarrow \cO\{\tau\}$. Let $\bar{\varphi}$ denote the reduction of $\varphi$. The Galois group $\op{G}_{K}$ acts on $^{\phi}K^{{\rm{sep}}}$ by $A$-linear automorphisms.

\begin{definition}A $\varphi$-lattice $\Lambda$ is a discrete, finitely generated and free $A$-submodule of $^{\phi}K^{{\rm{sep}}}$ and stable under $\op{G}_{K}$-action. The rank of $\Lambda$ is defined to be its rank as an $A$-module. A Drinfeld--Tate datum of rank $(r_1, r_2)$ is a pair $(\varphi, \Lambda)$, where $\varphi:A\rightarrow \cO\{\tau\}$ is a Drinfeld module of rank $r_1$ with good reduction and $\Lambda$ is a $\varphi$-lattice of rank $r_2$ over $A$. Two such datum $(\varphi, \Lambda)$ and $(\varphi', \Lambda')$ are isomorphic if there is an $\cO[\op{G}_K]$-isomorphism $\varphi\xrightarrow{\sim} \varphi'$ which induces an isomorphism of $A$-modules $\Lambda\xrightarrow{\sim} \Lambda'$. 
\end{definition}

\begin{theorem}[Drinfeld]\label{DT datum}
Given $\mathfrak{l}\in \Omega_A$ and a pair of positive integers $(r_1, r_2)$, there is a bijection between the following sets:
\begin{enumerate}
\item The set of isomorphism classes of Drinfeld modules $\phi: A\rightarrow \cO\{\tau\}$ of rank $r:=r_1+r_2$ with stable reduction and reduction rank $r_1$.
\item The set of isomorphism classes of $(\varphi,\Lambda)$ of rank $(r_1, r_2)$.
\end{enumerate}
\end{theorem}
\begin{proof}
For a proof of the result, see \cite[Section 6.2]{papibook}.
\end{proof}
The Galois representation of a Drinfeld module with stable reduction can be understood in terms of the associated Drinfeld--Tate datum, as the result below shows.
\begin{proposition}\label{Galois repn propn}
    Let $a\in A$ be a non-constant element and let $(\varphi, \Lambda)$ be Drinfeld--Tate datum of rank $(r_1, r_2)$ and let $\phi$ be the corresponding Drinfeld module of rank $r:=r_1+r_2$. Then, the following assertions hold:
    \begin{enumerate}
        \item there is a natural $G_{K}$-equivariant short exact sequence of $A$-modules:
\[0\longrightarrow \varphi[a]\longrightarrow \phi[a] \rightarrow \Lambda/a\Lambda \longrightarrow 0.\]
\item The exponential $e_\Lambda$ satisfies the relation 
\begin{equation}\label{e lambda equation}e_\Lambda(\phi_a(x))=\varphi_a (e_\Lambda(x))\end{equation}
and gives a Galois equivariant isomorphism of $A$-modules
\[e_\Lambda: \phi_a^{-1}\Lambda/\Lambda\xrightarrow{\sim} \varphi[a].\]
\item We have that
\[\phi_a(x)=ax\prod_{0\neq \pi\in \varphi_a^{-1}(\Lambda)/\Lambda} \left(1-\frac{x}{e_\Lambda(\pi)}\right).\]
    \end{enumerate}
    
\end{proposition}
\begin{proof}
See \cite[pp.355--356]{papibook} for the proofs of (1) and (2). Part (3) follows from \cite[Lemma 5.1.4 and Proposition 6.2.6]{papibook}.
\end{proof}
Consider the special case when $r_2=1$. Let $\gamma$ be a generator of $\Lambda$, i.e., 
\[\Lambda=\{\phi_a(\gamma)\mid a\in A\}.\]We then find that 
\begin{equation}
    e_\Lambda(x)=x\prod_{0\neq a\in A} \left(1-\frac{x}{\phi_a(\gamma)}\right).
\end{equation}
Then we have that 
\[\frac{e_\Lambda(x)}{x}=\sum_{i=0}^\infty a_i x^{q^i-1}, \] where for $i>0$, \begin{equation}\label{equation for a_i}a_i=(-1)^i\sum_{a_1, \dots, a_{q^i-1}\neq 0} \frac{1}{\phi_{a_1}(\gamma)\dots \phi_{a_{q^i-1}}(\gamma)}.\end{equation}

\section{Counting Drinfeld modules}\label{s 4}

\par We fix a finite field $\F_q$ and an integer $r\geq 2$. Let $|\cdot|_\infty$ be the absolute value at $\infty$, defined by $|a|=q^{\op{deg}_T(a)}$. Given $\vec{g}=(g_1, \dots, g_r) \in A^r$, we have the associated Drinfeld module defined by
\[\phi_T^{\vec{g}}=T+g_1\tau+g_2\tau^2+\dots+ g_r \tau^r. \]
We set $|\vec{g}|:=\op{max}\{|g_1|_{\infty} , \dots, |g_r|_{\infty}\}$, and for an integer $X>0$, set \[\cC_r(X):=\{\vec{g}=(g_1, \dots, g_r)\in A^r\mid g_r\neq 0, |\vec{g}|< q^X\}.\]
We note that the condition $|\vec{g}|<q^X$ can also be rephrased as $\op{deg}_T g_i< X$ for all $i=1, \dots, r$. It is easy to see that 
\begin{equation}\label{cC(X) count}\# \cC_r(X)=q^{rX}-q^{(r-1)X}.\end{equation}
Let $\cC_r$ be the set of all tuples $\vec{g}=(g_1, \dots, g_r)\in A^r$ such that $g_r\neq 0$. For $S$ a subset, we set \[S(X):=S\cap \cC_r(X)=\left\{\vec{g}\in S\mid |\vec{g}|< q^X\right\}.\] 
\begin{definition}
    The density of $S$ is defined as the following limit
    \[\mathfrak{d}(S):=\lim_{X\rightarrow \infty} \frac{\# S(X)}{\# \cC_r(X)},\]
    provided it exists. The \emph{upper} density $\overline{\mathfrak{d}}(S)$ (resp. \emph{lower} density $\underline{\mathfrak{d}}(S)$) is defined as above, upon replacing the limit with $\limsup$ (resp. $\liminf$). 
\end{definition}
Note that although $\mathfrak{d}(S)$ may not exist, the upper and lower densities $\overline{\mathfrak{d}}(S)$ and $\underline{\mathfrak{d}}(S)$ must exist. We say that $S$ has \emph{positive density} if $\underline{\mathfrak{d}}(S)>0$.

\par Let $\cF_r\subset \cC_r$ be the set vectors $\vec{g}=(g_1, \dots, g_r)\in A^r$ such that $g_r\neq 0$ and the $T$-adic Galois representation 
\[\rho^{\vec{g}}:=\hat{\rho}_{\phi^{\vec{g}}, T}: \op{G}_F\rightarrow \op{GL}_r\left(\F_q\llbracket T\rrbracket \right) \] is surjective. We shall show that $\mathfrak{d}(\cF_r)=1$, i.e., most Drinfeld modules have surjective $T$-adic Galois representation. Let $\bar{\rho}^{\vec{g}}:\op{G}_F\rightarrow \op{GL}_r(\F_q)$ be the mod-$T$ reduction of $\rho^{\vec{g}}$.
\par Let $\op{M}_r(\F_q)$ be the group of $r\times r$ matrices with entries in $\F_q$. Consider the following natural filtration on $G:=\op{GL}_r\left(\F_q\llbracket T\rrbracket \right)$ defined by $G^0:=G$, $G^i:=\op{Id}+T^i \op{M}_r(\F_q\llbracket T\rrbracket)$. We note that $G^{i+1}\subset G^i$ is a normal subgroup and 
\[G^{[i]}:=G^i/G^{i+1}\simeq \begin{cases}
    \op{GL}_r(\F_q) & \text{ if } i=0;\\
    \op{M}_r(\F_q) & \text{ if } i>0.\\
\end{cases}\]
Let $\mathcal{H}$ be a closed subgroup of $\op{GL}_r(\F_q\llbracket T\rrbracket)$, take $\mathcal{H}^i:=\mathcal{H}\cap G^i$, and $\mathcal{H}^{[i]}:=\mathcal{H}^i/\mathcal{H}^{i+1}$. The following criterion of Pink and R\"utsche gives a criterion for $\cH$ to equal $\op{GL}_r(\F_q\llbracket T \rrbracket)$.
\begin{proposition}\label{PR prop}
    Let $\mathcal{H}$ be a closed subgroup of $\op{GL}_r\left(\F_q\llbracket T \rrbracket\right)$ and assume that $q\geq 4$. Furthermore, suppose that the following conditions hold
    \begin{enumerate}
        \item $\mathcal{H}^{[0]}=\op{GL}_r\left(\F_q\right)$, 
          \item $\op{det}(\mathcal{H})=\F_q\llbracket T \rrbracket^\times$,
           \item $\mathcal{H}^{[1]}$ contains a non-scalar matrix.
    \end{enumerate}
    Then one has that $\mathcal{H}=\op{GL}_r(\F_q\llbracket T \rrbracket)$.
\end{proposition}
\begin{proof}
    This result is \cite[Proposition 4.1]{pinkrutsche}.
\end{proof}

We shall apply Proposition \ref{PR prop} to $\cH(\vec{g}):=\op{image}\left(\rho^{\vec{g}}\right)$. We identify $\mathcal{H}(\vec{g})^{[0]}$ with the image of $\bar{\rho}^{\vec{g}}$.
\begin{definition}
    We set $\cC_r'$ to be the set of $\vec{g}\in \cC_r$ for which the following conditions are satisfied
\begin{itemize}
   \item $\bar{\rho}^{\vec{g}}$ is surjective, 
          \item $\op{det}\rho^{\vec{g}}$ is surjective.
\end{itemize}
We take $\cC_r''$ to be the set of $\vec{g}\in \cC_r$ for which $\mathcal{H}(\vec{g})^{[1]}$ contains a non-scalar matrix.
\end{definition}

Proposition \ref{PR prop} asserts that 
\begin{equation}\label{set intersection property}
    \cF_r=\cC_r' \cap \cC_r'',
\end{equation}
and thus in order to show that $\mathfrak{d}(\cF_r)=1$, it suffices to show that $\mathfrak{d}(\cC_r')=1$ and $\mathfrak{d}(\cC_r'')=1$.

\par Let $\Delta$ be a non-zero element in $A$ and consider the associated Drinfeld module $\phi^\Delta$ of rank $1$ defined by $\phi_T^\Delta=T+\Delta \tau$. Let $\rho^\Delta$ denote the $T$-adic Galois representation \[\rho^\Delta=\hat{\rho}_{\phi^\Delta, T}:\op{G}_F\rightarrow \F_q\llbracket T\rrbracket^\times=\op{GL}_1\left(\F_q\llbracket T\rrbracket\right)\]
associated to $\phi^\Delta$. 
\begin{theorem}[Gekeler]\label{gekeler index thm}
    With respect to notation above, the index \[c_{\Delta}:=\left[\F_q\llbracket T\rrbracket^\times: \op{image}(\rho^\Delta)\right]\] is finite and is a divisor of $(q-1)$. 
\end{theorem}
\begin{proof}
    The result follows from \cite[Corollary 1.2]{Gekelertwistedcarlitz}.
\end{proof}

\begin{lemma}\label{detsurjlemma}
    With respect to notation above assume that $\bar{\rho}^{\vec{g}}$ is surjective, then $\vec{g}\in \cC_r'$.
\end{lemma}
\begin{proof}
    By the Weil pairing, $\rho^{\Delta}=\op{det} \rho^{\vec{g}}$, where $\Delta:=(-1)^{r-1} g_r$. Note that the kernel of the reduction mod-$T$ map 
    \[\pi_T:\F_q\llbracket T\rrbracket^\times\rightarrow \left(\F_q\llbracket T\rrbracket/(T)\right)^\times=\F_q^\times\] is $1+T \F_q\llbracket T\rrbracket$. For $i\geq 1$, the association $a\mapsto 1+T^i a$ gives rise to an isomorphism 
    \[\F_q\xrightarrow{\sim} \frac{(1+T^i\F_q\llbracket T \rrbracket)}{(1+T^{i+1}\F_q\llbracket T \rrbracket)}.\] In particular, $\op{ker}\pi_T$ is a pro-$p$ group. Since $\bar{\rho}^{\Delta}$ is surjective by assumption, it follows that the index $c_\Delta=\left[\F_q\llbracket T\rrbracket^\times: \op{image}(\rho^\Delta)\right]$ is of $p$-power order. Theorem \ref{gekeler index thm} then implies that $\rho^{\Delta}$ is surjective. Thus it follows that $\vec{g}$ satisfies the conditions defining $\cC_r'$.
\end{proof}

\begin{proposition}\label{cc prime =1}
    With respect to notation above, we have that $\mathfrak{d}(\cC_r')=1$.
\end{proposition}
\begin{proof}
\par The result follows from the proof of \cite[Theorem 3]{Galoisadd} and we provide details for the benefit of exposition. Consider the generic polynomial
\[F(x)=Tx+G_1x^q+G_2x^{q^2}+\dots+ G_r x^{q^r}\in \F_q(T, G_1,\dots, G_r)[x].\] Since $F'(x)=T\neq 0$, this polynomial is separable. Let \[Z(F)\subset \overline{\F_q(T, G_1,\dots, G_r)}\] be the set of zeros of $F$, which forms an $\F_q$-vector space of dimension $r$. This is evident since $F(x)$ is an $\F_q$-linear polynomial. Pick a basis $e_1,\dots, e_r$ of $Z(F)$. The splitting field of $F$ over $\F_q(T, G_1,\dots, G_r)$ is $L:=\F_q(T, e_1, \dots, e_r)$ and has transcendence degree $r$ over $\F_q(T)$. Hence, $e_1, \dots, e_r$ are algebraically independent over $\F_q(T)$. Thus, any element in $\op{GL}_r(\F_q)$ gives rise to an element of $\op{Gal}(L/\F_q(T))$, via the linear action on $e_1, \dots, e_r$. Since $\op{GL}_r(\F_q)$ permutes the elements $e_1, \dots, e_r$, it follows that 
\[\op{GL}_r(\F_q)\simeq \op{Gal}(L/\F_q(T)).\] Via the permutation action, we simply identify $\op{Gal}(L/\F_q(T))$ with $\op{GL}_r(\F_q)$. Given a tuple $\vec{g}=(g_1, \dots, g_r)\mathcal{C}_r(X)$, one has a polynomial $F_{\vec{g}}(x)=Tx+g_1x^q+\dots +g_r x^{q^r}\in \F_q(T)[x]$. Let $G(\vec{g})\subset \op{GL}_r(\F_q)$ be its Galois group and denote by $B_r(X)$ the number of $\vec{g}\in \mathcal{C}_r(X)$ such that $G(\vec{g})\neq  \op{GL}_r(\F_q)$, or equivalently, $\bar{\rho}^{\vec{g}}$ is not surjective. Then, according to \cite[Corollary 3.5]{HITFF}, one finds that 
\[\frac{B_r(X)}{q^{rX}}=O(Xq^{-X/2}),\] a quantity which goes to $0$ as $X\rightarrow \infty$. Lemma \ref{detsurjlemma} then implies that if $\bar{\rho}^{\vec{g}}$ is surjective, then $\vec{g}\in \cC_r'$. Thus it follows that \[\mathfrak{d}(\cC_r')\leq 1-\lim_{X\rightarrow \infty} \frac{B_r(X)}{\#\mathcal{C}_r(X)}= 1-\lim_{X\rightarrow \infty} \frac{B_r(X)}{q^{rX}(1-q^{-X})}=1.\]
\end{proof}

In section \ref{s 6}, we show that $\mathfrak{d}(\cC_r'')=1$.

\section{Congruence conditions}
For $\vec{g}\in \cC_r$, recall that $\cH(\vec{g})$ is the image of $\rho^{\vec{g}}$. Note that \[\cH(\vec{g})^i=\rho^{\vec{g}}\left(\op{G}_{F(\phi[T^i])}\right)\] and it is easy to see that $\rho^{\vec{g}}$ induces an isomorphism 
\[\op{Gal}\left(F(\phi[T^{i+1}])/F(\phi[T^i])\right)\xrightarrow{\sim} \cH(\vec{g})^{[i]}.\] 
The set $\cC_r''$ was defined to consist of those $\vec{g}\in \cC_r$ for which $\cH(\vec{g})^{[1]}$ contains a non-scalar element. We show that this set can be understood in terms of congruence conditions at all primes $\fl\in \Omega_A$ such that $\fl\neq (T)$. 
\begin{definition}Suppose $r\geq 2$ and let $\Pi_r$ consist of $\vec{g}\in \cC_r$ such that there is a prime $\fl\in \Omega_A$ with $\fl\neq (T)$, such that the following conditions hold for $v:=v_\fl$: 
\begin{itemize}
    \item $v(g_{r-1})=0$, 
    \item $p\nmid v(g_{r})$.
\end{itemize}
\end{definition}
We show in this section that $\Pi_r$ is contained in $\cC_r''$ and in the next section, it is shown that $\Pi_r$ has density $1$. Some of the constructions are inspired by arguments of Zywina \cite{zywina2011drinfeld} and Chen \cite{Chen2}. For $\vec{g}\in \Pi_r$, and $\fl$ a prime for which the valuations of $g_{r-1}$ and $g_r$ are given as above, it is easy to see that $\phi^{\vec{g}}$ has stable reduction at $\mathfrak{l}$ with reduction rank $(r-1)$. We set $v:=v_{\fl}$ throughout. By Theorem \ref{DT datum}, there is a Drinfeld-Tate datum $(\varphi, \Lambda)$ that corresponds to $\phi_{\mathfrak{l}}=\phi_{\mathfrak{l}}^{\vec{g}}$ (the completion of $\phi^{\vec{g}}$ at $\fl$). Here, $\varphi$ has good reduction and rank $(r-1)$ and $\Lambda$ is a $\varphi$-lattice of rank $1$. Proposition \ref{Galois repn propn} implies that for all $n\geq 1$, there is a short exact sequence of $\op{G}_{F_{\mathfrak{l}}}$-modules
\begin{equation}\label{T^n exact sequence}0\longrightarrow \varphi[T^n]\longrightarrow \phi[T^n] \xrightarrow{\varphi_{T^n}}\Lambda/T^n \Lambda \longrightarrow 0.\end{equation}
We note that since $\varphi$ has good reduction at $\fl\neq (T)$, the action of $\op{G}_{F_{\fl}}$ on $\varphi[T^n]$ is unramified. Taking the inverse limit with respect to multiplication by $T$ maps, we obtain the following exact sequence of $A_{(T)}[\op{G}_{F_{\fl}}]$-modules
\begin{equation}\label{T t exact sequence}0\rightarrow \mathbb{T}_{(T)}(\varphi) \rightarrow \mathbb{T}_{(T)}(\phi)\rightarrow \Lambda\otimes_{A} A_{(T)}\rightarrow 0.\end{equation}
Set $\rho$ to denote the $T$-adic Galois representation $\rho^{\vec{g}}$ (associated to $\phi^{\vec{g}}$). For $n\geq 1$, denote by $\rho_n:\op{G}_F\rightarrow \op{GL}_r\left(A/(T^n)\right)$ the reduction of $\rho$ modulo $(T^n)$. The inertia group acts trivially on $\mathbb{T}_{(T)}(\varphi)$. It follows from \eqref{T t exact sequence} that $\rho_{|\op{I}_{\mathfrak{l}}}$ consists of matrices of the form 
\[\left(\begin{array}{ccccc}1 & 0 & \cdots & 0 & \ast \\ & \ddots &  & \vdots & \vdots \\ &  & \ddots & 0 & \ast \\ &  &  & 1 &  \ast \\ &  &  & &  \chi. \end{array}\right)\]
 Here, \[\chi:\op{I}_{\fl}\rightarrow \op{Aut}(\Lambda\otimes_{A} A_{(T)})\xrightarrow{\sim} A_{(T)}^\times\] coincides with the determinant of $\rho_{|\op{I}_{\mathfrak{l}}}$.

\par Let $\cU$ (resp. $\cU_n$) be the subgroup of $\op{GL}_r(A_{(T)})$ (resp. $\op{GL}_r(A/(T^n)$) consisting of matrices of the form
\[\left(\begin{array}{ccccc}1 & 0 & \cdots & 0 & \ast \\ & \ddots &  & \vdots & \vdots \\ &  & \ddots & 0 & \ast \\ &  &  & 1 &  \ast \\ &  &  & &  \ast. \end{array}\right)\] Observe that the association
\[\left((m_1, \dots, m_{r-1}), m_r\right)\mapsto \left(\begin{array}{ccccc}1 & 0 & \cdots & 0 & m_1 \\ & \ddots &  & \vdots & \vdots \\ &  & \ddots & 0 & m_{r-2} \\ &  &  & 1 &  m_{r-1} \\ &  &  & &  m_r. \end{array}\right)\]
defines an isomorphism
\[\Phi_n: \left(A/(T^n)\right)^{r-1}\times \left(A/(T^n)\right)^\times \xrightarrow{\sim} \cU_n.\]
Let $\mathcal{W}$ defined as the kernel of the mod-$T$ reduction map $\cU_2\xrightarrow{\op{mod}_T} \cU_1$. Identify $\mathcal{W}$ with the kernel of the reduction map 
\[\left(A/(T^2)\right)^{r-1}\times \left(A/(T^2)\right)^\times \xrightarrow{\op{mod}_T} \left(A/(T)\right)^{r-1}\times \left(A/(T)\right)^\times .\]
Thus in particular, $\mathcal{W}$ is isomorphic to $(A/(T))^r$. 
\begin{proposition}\label{prop rho2 contains U2}
    With respect to notation above, $\rho_2(\op{I}_{\fl})$ has nontrivial intersection with $\mathcal{W}$.
\end{proposition}
It follows from the above Proposition that $\vec{g}\in \cC_r''$. We postpone the proof till the end of this section. Setting $\mathfrak{J}:=(\varphi_{T^2})^{-1}(\Lambda)/\Lambda$, it follows from Proposition \ref{Galois repn propn} that $e_\Lambda$ induces a Galois equivariant isomorphism $\mathfrak{J}\xrightarrow{\sim} \phi[T^2]$. Let $\gamma$ be a generator of $\Lambda$ and $z\in F_{\mathfrak{l}}^{\op{sep}}$ be such that $\varphi_{T^2}(z)=\gamma$. We consider the maximal unramified extension $F_{\mathfrak{l}}^{\op{nr}}$ of $F_{\mathfrak{l}}$. Note that
\[\rho_2(\op{I}_{\mathfrak{l}})\simeq \op{Gal}\left(F_{\mathfrak{l}}^{\op{nr}}(\phi[T^2])/F_{\mathfrak{l}}^{\op{nr}}\right)\simeq \op{Gal}\left(F_{\mathfrak{l}}^{\op{nr}}(\mathfrak{J})/F_{\mathfrak{l}}^{\op{nr}}\right)\]
and that $z$ is contained in $F_{\mathfrak{l}}^{\op{nr}}(\mathfrak{J})$. Therefore, we find that $F_{\mathfrak{l}}^{\op{nr}}(z)\subseteq F_{\mathfrak{l}}^{\op{nr}}(\mathfrak{J})$. We shall calculate the ramification index $z$ and thus obtain some clarity of the power of $q$ that divides $[F_{\mathfrak{l}}^{\op{nr}}(\phi[T^2]): F_{\mathfrak{l}}^{\op{nr}}]$.

\par It follows from Proposition \ref{Galois repn propn} that 
\[\phi_{T^2}(x)=T^2 x\prod_{0\neq \pi\in \mathfrak{J}} \left(1-\frac{x}{e_\Lambda(\pi)}\right).\]
Let us compare the valuations of the leading coefficients of the left and right hand sides of the equation above. One has that
\[g_r^{1+q^r}=(-1)^{\# \mathfrak{J}} T^2 \left(\prod_{0\neq \pi\in \mathfrak{J}} e_{\Lambda}(\pi)\right)^{-1}.\]
Let $w_1, \dots, w_{r-1}$ be an $A/(T^2)$-basis for $\varphi[T]$. From the exactness of the sequence \eqref{T^n exact sequence}, we find that $w_1, \dots, w_{r-1}, z$ is an $A/(T^2)$-basis for $\phi_{\fl}[T^2]$. One finds that 
\begin{equation}\label{long equation}-(1+q^r)v(g_r)=\sum_{a_1, \dots, a_{r-1}, b}v\left(e_\Lambda(a_1 w_1+a_2 w_2+\dots+a_{r-1} w_{r-1} +b z)\right),\end{equation} where the sum is over tuples $(a_1, \dots, a_{r-1}, b)\in (A/(T^2))^r$ that are not identically $0$.
\begin{lemma}
    The following assertions hold;
    \begin{enumerate}
        \item $v(\gamma)<0$, 
        \item $v(z)<0$, 
        \item $v(e_\Lambda(z))=v(z)$.
    \end{enumerate}
\end{lemma}
\begin{proof}
    \par Part (1) follows from \cite[Example 6.2.2 p.353]{papibook}. 
    
    \par Since every coefficient of $\varphi_{T^2}(x)$ has non-negative valuation and $\varphi_{T^2}(z)=\gamma$, it follows that $v(z)<0$. This proves (2). 
    
    \par We proceed with the proof of part (3). Note that 
    \[e_{\Lambda}(z)=z+\sum_{i=1}^\infty a_i z^{q^i},\] where according to \eqref{equation for a_i}, 
    \[a_i=(-1)^i\sum_{a_1, \dots, a_{q^i-1}\neq 0} \frac{1}{\phi_{a_1}(\gamma)\dots \phi_{a_{q^i-1}}(\gamma)}.\]
    We thus find that that $v(a_i)\geq -(q^i-1) v(\gamma)$, and as a result,
    \[v(a_i z^{q^i})\geq -(q^i-1)v(\gamma)+q^{i-2(r-1)}v(\gamma)\geq 0. \]
    Having proven that \[v\left(\sum_{i=1}^\infty a_i z^{q^i}\right)\geq 0\] and $v(z)<0$, we deduce that $v(e_\Lambda(z))=v(z)$. This completes the proof of (3).
\end{proof}

\begin{lemma}\label{tiny lemma}
    Let $w\in F_{\fl}^{\op{sep}}$ be such that $v(w)\geq 0$. Then we find that $v(e_\Lambda(w))=v(w)$. 
\end{lemma}
\begin{proof}
    We write 
    \[e_\Lambda(w)=w+\sum_{i=1} a_i w^{q^i},\] note that $v(a_i)>0$ and $v(w)\geq 0$. Therefore, it follows that $v(e_\Lambda(w))=v(w)$.
\end{proof}

\begin{lemma}\label{v equals technical lemma}
    With respect to notation above, we have that 
    \[\begin{split} & v\left(e_\Lambda(a_1 w_1+a_2 w_2+\dots+a_{r-1} w_{r-1} +b z)\right) \\ = &\begin{cases}
        q^{(r-1)i}v(z) & \text{ if }b\neq 0\text{ and }\op{deg}_T(b)=i.\\
        v(a_1 w_1+a_2 w_2+\dots+a_{r-1} w_{r-1})& \text{ otherwise. }
    \end{cases}
    \end{split}\]
\end{lemma}
\begin{proof}
    \par First we consider the case when $b\neq 0$ and write $b=c+dT$, where $c, d\in \F_q$. One finds that 
\[\begin{split} & e_\Lambda(a_1 w_1+a_2 w_2+\dots+a_{r-1} w_{r-1} +b z) \\ 
=& \sum_{i=1}^{r-1} \varphi_{a_i} \left(e_{\Lambda}(w_i)\right)+\varphi_b \left(e_{\Lambda}(z)\right).\end{split}\]
Note that $w_i$ is a root of $\varphi_{T^2}(x)$. We write \[\varphi_T(\tau)=T+b_1 \tau+ b_2 \tau^2+\dots+ b_{r-1}\tau^{r-1}\] and note that $v(b_{r-1})=0$ since $\varphi$ has good reduction. We find that 
\[\varphi_{T^2}=\varphi_T\circ \varphi_T=T^2+c_1\tau+\dots +c_{2(r-1)}\tau^{2(r-1)},\] where $c_{2(r-1)}=b_{r-1}^{1+q^{r-1}}$. In particular, $v(c_{2(r-1)})=0$ and all slopes of the Newton polygon of $\phi_{T^2}(x)$ are $\leq 0$. We deduce therefore that all solutions to $\varphi_{T^2}(x)=0$ have non-negative valuation, in particular, $v(w_i)\geq 0$ for $1\leq i\leq r-1$. In particular, we find that $v(w_i)\geq 0$. It follows from Lemma \ref{tiny lemma} that $v(e_{\Lambda}(w_i))=v(w_i)\geq 0$. We find that
\[\varphi_b(x)=cx+d\varphi_T(x)=bx+db_1x^q+\dots +d b_{(r-1)} x^{q^{(r-1)}},\]
and therefore, 
\[v\left(\varphi_b(e_\Lambda(z))\right)=q^{(r-1)}v(e_\Lambda(z))=\begin{cases}
    q^{(r-1)}v(z) &\text{ if }d\neq 0;\\
    v(z) &\text{ if }d=0.
\end{cases}\]
Therefore, we have shown that if $b\neq 0$, 
\[v\left(e_\Lambda(a_1 w_1+a_2 w_2+\dots+a_{r-1} w_{r-1} +b z)\right)=v\left(\varphi_b(e_\Lambda(z))\right)=q^{(r-1)i}v(z),\]
where $i:=\op{deg}_T(b)$.
\par When $b=0$, it follows from Lemma \ref{tiny lemma} that 
\[v\left(e_\Lambda(a_1 w_1+a_2 w_2+\dots+a_{r-1} w_{r-1})\right)=v(a_1 w_1+a_2 w_2+\dots+a_{r-1} w_{r-1}).\]
\end{proof}
We conclude this section with the proof of Proposition \ref{prop rho2 contains U2}.

\begin{proof}[Proof of Proposition \ref{prop rho2 contains U2}]
    Applying Lemma \ref{v equals technical lemma}, we simplify the right hand side of \eqref{long equation}. First, we note that
    \[\begin{split}& \sum_{a_1, \dots, a_{r-1}, b; b\neq 0}v\left(e_\Lambda(a_1 w_1+a_2 w_2+\dots+a_{r-1} w_{r-1} +b z)\right) \\
=& q^{2(r-1)}  \left( \sum_{b, \op{deg} b=0} v(z)+\sum_{b, \op{deg} b=1} q^{(r-1)}v(z)\right) \\
=& q^{2(r-1)} v(z) \left( (q-1)+q(q-1)q^{(r-1)}\right) \\
=& q^{2(r-1)}(q-1)v(z)(1+q^{r}).\end{split}\]
On the other hand, 
  \[\begin{split}& \sum_{(a_1, \dots, a_{r-1})\neq 0}v\left(e_\Lambda(a_1 w_1+a_2 w_2+\dots+a_{r-1} w_{r-1})\right) \\
  =& \sum_{(a_1, \dots, a_{r-1})\neq 0}v(a_1 w_1+a_2 w_2+\dots+a_{r-1} w_{r-1}) \\
=& v\left(\prod_{(a_1, \dots, a_{r-1})\neq 0}(a_1 w_1+a_2 w_2+\dots+a_{r-1} w_{r-1})\right)\\
=& v\left(\prod_{0\neq w \in \phi[T^2]}w\right)=v(T^2)=0.\end{split}\]
For the final relation, we use the fact that $\prod_{0\neq w \in \varphi[T^2]}w$ is the constant term of $\varphi_{T^2}(x)/x$ and therefore equals $T^2$.
Thus the right hand side of \eqref{v equals technical lemma} equals \[q^{2(r-1)}(q-1)v(z)(1+q^{r}).\] Thus, we have shown that 
\[v(z)=-\frac{v(g_r)}{(q-1)q^{2(r-1)}}.\]
Since $p\nmid v(g_r)$, we find that $q^{2(r-1)}$ divides the ramification index of $F_{\fl}(z)/F_{\fl}$. This implies that $q^{2(r-1)}$ divides $\#\rho_2(\op{I}_{\fl})$. Note that $\rho_2(\op{I}_{\fl})$ is contained in $\cU_2$ and $\# \cU_2=(q-1) q^{2r-1}$. Also note that $\# \cU_1=(q-1) q^{r-1}$ and $\# \mathcal{W}=q^{r}$. We deduce that $q^{r-1}$ divides $\# \left(\rho_2(\op{I}_{\fl})\cap \mathcal{W}\right)$. This proves the result.
\end{proof}

\begin{corollary}\label{ Pi in cc}
    One has the following inclusion of sets:
    \[\Pi_r\subseteq \cC_r''.\]
\end{corollary}
\begin{proof}
    Let $\vec{g}\in \Pi_r$, then by assumption there is a prime $\fl\in \Omega_A$ such that $\fl\neq (T)$ and 
    \begin{itemize}
    \item $v(g_{r-1})=0$, 
    \item $p\nmid v(g_{r})$.
\end{itemize}
Proposition \ref{prop rho2 contains U2} implies that $\rho_2(\op{I}_{\fl})$ contains a nontrivial element in $\mathcal{W}$. Note that $\mathcal{W}$ does not contain any scalar matrices. Therefore, it follows that $\cH(\vec{g})^{[1]}$ contains a non-scalar matrix. In other words, $\vec{g}\in \cC_r''$.
\end{proof}
\section{Density results}\label{s 6}

\par In this section we prove that $\mathfrak{d}(\cF_r)=1$. We begin by showing that $\Pi_r$ defined in the previous section has density $1$.

\begin{proposition}\label{Pi r has density 1}
    With respect to notation above, $\mathfrak{d}(\Pi_r)=1$.
\end{proposition}
\begin{proof}

\par Let $\Omega_r$ be the subset of $\cC_r$ consisting of $\vec{g}=(g_1, \dots, g_r)$ such that for all primes $\mathfrak{l}\in \Omega_A\setminus (T) $, one has that $ v(g_r)=0$ or $v(g_r)\geq p$. It is clear that the complement of $\Pi_r$ is contained in $\Omega_r$. We show that $\mathfrak{d}(\Omega_r)=0$ and deduce that $\mathfrak{d}(\Pi_r)=1$. Let $\cC_r(\mathfrak{l})$ consist of all tuples in $(A/\mathfrak{l}^p)^r$ and consider the $\fl^p$-reduction map 
\[\pi_{\fl} :\cC_r\rightarrow \cC_r(\mathfrak{l}).\]
Let $\Omega_r(\mathfrak{l}):=\pi_{\fl}\left(\Omega_r\right)$, consisting of all tuples $(\bar{g}_1, \bar{g}_2, \dots, \bar{g}_r)$ such that $\mathfrak{l}\nmid \bar{g}_{r}$ or $\bar{g}_r=0$. Therefore, we find that $\# \Omega_r(\fl)=q_{\fl}^{p(r-1)}\left(q_{\fl}^p-q_{\fl}^{p-1}+1\right)$, where $q_\fl:=q^{\op{deg}(\fl)}$. Define
\[\mathfrak{d}(\fl):=\frac{\# \Omega_r(\fl)}{\# \cC_r(\fl)}=\left(1-\frac{1}{q_{\fl}}+\frac{1}{q_{\fl}^p}\right)\]
and given a finite set of primes $S\subset \Omega_A\backslash (T)$ set $\mathfrak{d}_S:=\prod_{\fl\in S} \mathfrak{d}(\fl)$. Define $\Omega_r^S$ to be the set of all $\vec{g}\in \cC_r$ such that for all primes $\fl\in S$,  \[v(g_r)=0\text{ or }v(g_r)\geq p.\]
Since $\Omega_r^S$ is defined by finitely many congruence conditions, it is easy to see that 
\[\Omega_r^S\sim \mathfrak{d}_S q^{rX}. \]
On the other hand, $\Omega_r\subset \Omega_r^S$ and thus for all $S$, 
\[\overline{\mathfrak{d}}(\Omega_r)\leq \mathfrak{d}(\Omega_r^S)=\left(\prod_{\fl\in S} \mathfrak{d}(\fl)\right).\]
Since $S$ is an arbitrary finite subset of $\Omega_A\backslash (T)$, it follows that 
\[\overline{\mathfrak{d}}(\Omega_r)\leq \prod_{\fl\in \Omega_A\backslash (T)}\mathfrak{d}(\fl)\leq \prod_{\fl\in \Omega_A\backslash (T)}\left(1-\frac{1}{q_{\fl}}+\frac{1}{q_{\fl}^p}\right). \]
Lemma \ref{last lemma} shows that 
\[\prod_{\fl\in \Omega_A\backslash (T)}\left(1-\frac{1}{q_{\fl}}+\frac{1}{q_{\fl}^p}\right)=0\] and thus it follows that $\mathfrak{d}(\Omega_r)=0$. This implies that $\mathfrak{d}(\Pi_r)=1$.
\end{proof}
\begin{lemma}\label{last lemma}
    With respect to notation above, 
    \[\prod_{\fl\in \Omega_A}\left(1-\frac{1}{q_{\fl}}+\frac{1}{q_{\fl}^p}\right)=0. \]
\end{lemma}
\begin{proof}
    Note that \[\prod_{\fl}\left(1-\frac{1}{q_{\fl}}+\frac{1}{q_{\fl}^p}\right)=\prod_{n=1}^\infty \left(1-\frac{1}{q^n}+\frac{1}{q^{np}}\right)^{c_n},\] where $c_n$ is the number of irreducible polynomial of degree $n$ over $\F_q$. This product converges to $0$ if and only if 
    \[\sum_{n=1}^\infty c_n\op{log} \left(1-\frac{1}{q^n}+\frac{1}{q^{np}}\right)\] goes to $-\infty$. By the limit comparison test, this is equivalent to 
    \[\sum_n \frac{c_n}{q^n} \] goes to $+\infty$. We have that 
    \[c_n=q^n/n +O(q^{n/2}/n)\]
    \cite[Theorem 2.2]{Rosen} the result follows.
\end{proof}

\begin{theorem}\label{main thm end}
    Let $r\geq 2$, then $\cF_r$ has density $1$.
\end{theorem}
\begin{proof}
    Noting that $\cF_r=\cC_r'\cap \cC_r''$ and that Proposition \ref{cc prime =1} asserts that $\mathfrak{d}(\cC_r')=1$. Thus, in order to complete the proof, it suffices to show that $\mathfrak{d}(\cC_r'')=1$. Corollary \ref{ Pi in cc} asserts that the set $\Pi_r$ introduced in the previous section is shown to be contained in $\cC_r''$. Proposition \ref{Pi r has density 1} shows that $\Pi_r$ has density $1$. This completes the proof.
\end{proof}
\bibliographystyle{alpha}
\bibliography{references}

@article{zywina2011drinfeld,
  title={Drinfeld modules with maximal Galois action on their torsion points},
  author={Zywina, David},
  journal={arXiv preprint arXiv:1110.4365},
  year={2011}
}

@article {pinkrutsche,
    AUTHOR = {Pink, Richard and R\"{u}tsche, Egon},
     TITLE = {Adelic openness for {D}rinfeld modules in generic
              characteristic},
   JOURNAL = {J. Number Theory},
  FJOURNAL = {Journal of Number Theory},
    VOLUME = {129},
      YEAR = {2009},
    NUMBER = {4},
     PAGES = {882--907},
}

@book {papibook,
    AUTHOR = {Papikian, Mihran},
     TITLE = {Drinfeld modules},
    SERIES = {Graduate Texts in Mathematics},
    VOLUME = {296},
 PUBLISHER = {Springer, Cham},
      YEAR = {[2023] \copyright 2023},
     PAGES = {xxi+526},
}

@article{Chen2,
  title={Surjectivity of the adelic Galois representation associated to a Drinfeld module of prime rank},
  author={Chen, Chien-Hua},
  journal={arXiv preprint arXiv:2111.04234},
  year={2021}
}

@article {Gekelertwistedcarlitz,
    AUTHOR = {Gekeler, Ernst-Ulrich},
     TITLE = {The {G}alois image of twisted {C}arlitz modules},
   JOURNAL = {J. Number Theory},
  FJOURNAL = {Journal of Number Theory},
    VOLUME = {163},
      YEAR = {2016},
     PAGES = {316--330},
}

@article {raytadic,
    AUTHOR = {Ray, Anwesh},
     TITLE = {The {$T$}-adic {G}alois representation is surjective for a
              positive density of {D}rinfeld modules},
   JOURNAL = {Res. Number Theory},
  FJOURNAL = {Research in Number Theory},
    VOLUME = {10},
      YEAR = {2024},
    NUMBER = {3},
     PAGES = {Paper No. 56, 12},
}

@article {zywina,
    AUTHOR = {Zywina, David},
     TITLE = {On the surjectivity of {${\rm mod}\,\ell$} representations
              associated to elliptic curves},
   JOURNAL = {Bull. Lond. Math. Soc.},
  FJOURNAL = {Bulletin of the London Mathematical Society},
    VOLUME = {54},
      YEAR = {2022},
    NUMBER = {6},
     PAGES = {2404--2417},
}

@article {Mazurmodcurves,
    AUTHOR = {Mazur, B.},
     TITLE = {Modular curves and the {E}isenstein ideal},
      NOTE = {With an appendix by Mazur and M. Rapoport},
   JOURNAL = {Inst. Hautes \'{E}tudes Sci. Publ. Math.},
  FJOURNAL = {Institut des Hautes \'{E}tudes Scientifiques. Publications
              Math\'{e}matiques},
    NUMBER = {47},
      YEAR = {1977},
     PAGES = {33--186 (1978)},
}

@article {Sutherland,
    AUTHOR = {Sutherland, Andrew V.},
     TITLE = {Computing images of {G}alois representations attached to
              elliptic curves},
   JOURNAL = {Forum Math. Sigma},
  FJOURNAL = {Forum of Mathematics. Sigma},
    VOLUME = {4},
      YEAR = {2016},
     PAGES = {Paper No. e4, 79},
}

@book {Rosen,
    AUTHOR = {Rosen, Michael},
     TITLE = {Number theory in function fields},
    SERIES = {Graduate Texts in Mathematics},
    VOLUME = {210},
 PUBLISHER = {Springer-Verlag, New York},
      YEAR = {2002},
}

@article {Galoisadd,
    AUTHOR = {Bary-Soroker, Lior and Entin, Alexei and McKemmie, Eilidh},
     TITLE = {Galois groups of random additive polynomials},
   JOURNAL = {Trans. Amer. Math. Soc.},
  FJOURNAL = {Transactions of the American Mathematical Society},
    VOLUME = {377},
      YEAR = {2024},
    NUMBER = {3},
     PAGES = {2231--2259},
}

@incollection {HITFF,
    AUTHOR = {Bary-Soroker, Lior and Entin, Alexei},
     TITLE = {Explicit {H}ilbert's irreducibility theorem in function
              fields},
 BOOKTITLE = {Abelian varieties and number theory},
    SERIES = {Contemp. Math.},
    VOLUME = {767},
     PAGES = {125--134},
 PUBLISHER = {Amer. Math. Soc., [Providence], RI},
      YEAR = {[2021] \copyright 2021},
}

@article {Drinfeldoriginal,
    AUTHOR = {Drinfeld, V. G.},
     TITLE = {Elliptic modules},
   JOURNAL = {Mat. Sb. (N.S.)},
  FJOURNAL = {Matematicheski\u{\i} Sbornik. Novaya Seriya},
    VOLUME = {94(136)},
      YEAR = {1974},
     PAGES = {594--627, 656},
}

@article {Jones,
    AUTHOR = {Jones, Nathan},
     TITLE = {Almost all elliptic curves are {S}erre curves},
   JOURNAL = {Trans. Amer. Math. Soc.},
  FJOURNAL = {Transactions of the American Mathematical Society},
    VOLUME = {362},
      YEAR = {2010},
    NUMBER = {3},
     PAGES = {1547--1570},
}

@article{duke1997elliptic,
  title={Elliptic curves with no exceptional primes},
  author={Duke, William},
  journal={Comptes rendus de l'Acad{\'e}mie des sciences. S{\'e}rie 1, Math{\'e}matique},
  volume={325},
  number={8},
  pages={813--818},
  year={1997}
}
\end{document}